\documentclass[10pt]{amsart}
\usepackage{amsmath,amssymb,amsthm,enumerate,graphicx,epstopdf,caption,subcaption}
\newtheorem{thm}{Theorem}
\newtheorem{prop}{Proposition}
\theoremstyle{remark}
\newtheorem*{rem}{Remark}

\begin{document}

\title[A Quantum Gauss-Bonnet Theorem]{A Quantum Gauss-Bonnet Theorem}
\author{Taylor Friesen}

\begin{abstract}
In \cite{lp}, Lanzat and Polyak introduced a polynomial invariant of generic curves in the plane as a quantization of Hopf's Umlaufsatz, and showed that Arnold's $J^+$ invariant could be derived from their polynomial, leading to an integral formula for $J^+$.
Here we extend their invariant to the case of homologically trivial generic curves in closed oriented surfaces with Riemannian metric. The resulting invariant turns out to be a quantization of a new formula for the rotation number, which can be viewed as a form of the Gauss-Bonnet Theorem. We show that $J^+$ can be calculated from the generalized invariant when the Euler characteristic of the surface is nonzero, thereby obtaining an integral formula for $J^+$ for homologically trivial curves in oriented surfaces with nonzero Euler characteristic.
\end{abstract}

\maketitle

\section{Introduction}
We begin by considering an immersed curve $\Gamma: S^1 \to \mathbb R^2$. The \emph{rotation number} $\mathrm{rot}(\Gamma)$ is defined to be the (signed) number of turns made by the tangent vector as we travel along $\Gamma$, or more formally the degree of $\Gamma$'s Gauss map. The \emph{index} $\mathrm{ind}_{\Gamma}(p)$ of a point $p \in \mathbb R^2 \setminus \Gamma$ with respect to $\Gamma$ is the (signed) number of times $\Gamma$ revolves around $p$, or more formally the degree of the map $S^1 \to S^1$ given by
\[ t \mapsto \frac{\Gamma(t) - p}{\|\Gamma(t) - p\|} \]
\begin{rem}The reader should beware that in some sources the rotation number is referred to as the index of the curve, and that the term ``winding number'' may be used to refer either to the index or the rotation number.
\end{rem}
It is a basic fact of differential geometry that $\Gamma$'s rotation number can be calculated by integrating its curvature and dividing by $2\pi$, i.e.
\begin{equation}\label{HU-multiplicities}
\frac{1}{2\pi} \int_{S^1} k(t) \, dt = \mathrm{rot}(\Gamma)
\end{equation}
This formula is referred to as Hopf's Umlaufsatz by Lanzat and Polyak in \cite{lp}, although other sources, including the paper in which Hopf introduced the term, \cite{hopf}, use ``Umlaufsatz'' to refer to the more specific case where $\Gamma$ is a simple curve (i.e. $\Gamma$ has no self-intersections), in which case
\begin{equation}\label{HU-original}
\frac{1}{2\pi}\int_{S^1} k(t) \, dt = \pm 1
\end{equation}
In this sense Equation (\ref{HU-multiplicities}) can be seen as an ``Umlaufsatz with multiplicities''. Later, when we consider curves in more general surfaces, we will see that we need a ``Gauss-Bonnet theorem with multiplicities'' in the same vein.

While the rotation number can be calculated for any immersed curve in the plane and is invariant under all regular homotopies, Arnold's invariants $J^\pm$ and $St$ are defined only for \emph{generic} curves in the plane---curves which have a finite number of transverse double points as their only self-intersections---and are invariant under regular homotopies within this narrower class.

It is well-known that if two generic curves are regular homotopic in the space of immersed curves, then one can be transformed into the other by a series of homotopies in the space of generic curves together with a finite number of \emph{self-tangency} and \emph{triple-point} moves, shown in Figure \ref{moves}. Self-tangency moves can be further classified into direct and opposite self-tangencies, according to whether the two involved arcs of the curve are pointing in the same direction or in opposite directions (see Figure \ref{st-classification}). Arnold's $J^+$ invariant can be understood as counting direct self-tangency moves: It increases by 2 under direct self-tangency moves and is unchanged under opposite self-tangency moves and triple-point moves. The $J^-$ and $St$ invariants play analogous roles for opposite self-tangency moves and triple-point moves, respectively \cite{j+}.

\begin{figure}
\centering
\begin{subfigure}{0.4\textwidth}
	\includegraphics[width=\textwidth, trim = 1cm 18cm 4cm 3cm, clip]{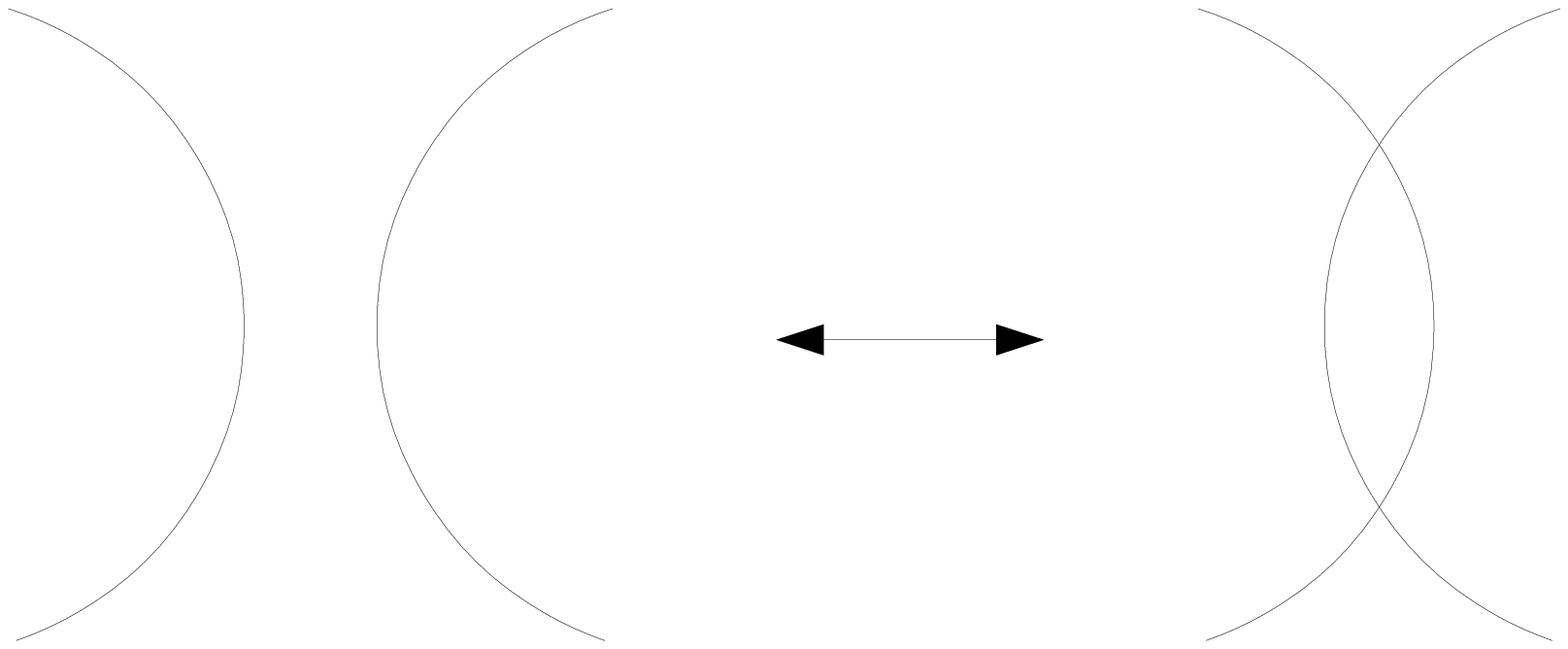}
	\caption{A self-tangency move} 
\end{subfigure}
\hspace{2cm}
\begin{subfigure}{0.4\textwidth}
	\includegraphics[width=\textwidth, trim = 0cm 19cm 1cm 0, clip]{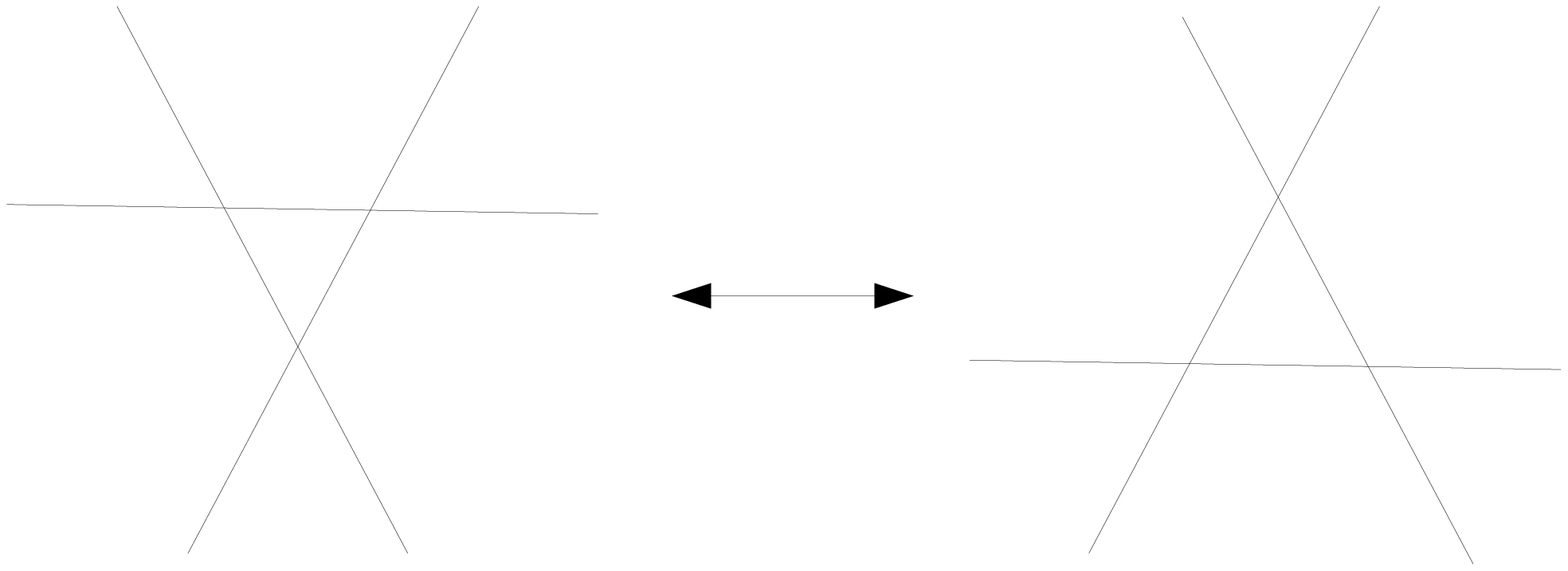}
	\caption{A triple-point move}
\end{subfigure}
\caption{The self-tangency and triple-point moves}
\label{moves}
\end{figure}

\begin{figure}
\centering
\begin{subfigure}{0.4\textwidth}
	\includegraphics[width=\textwidth, trim = 1cm 18cm 3cm 3cm, clip]{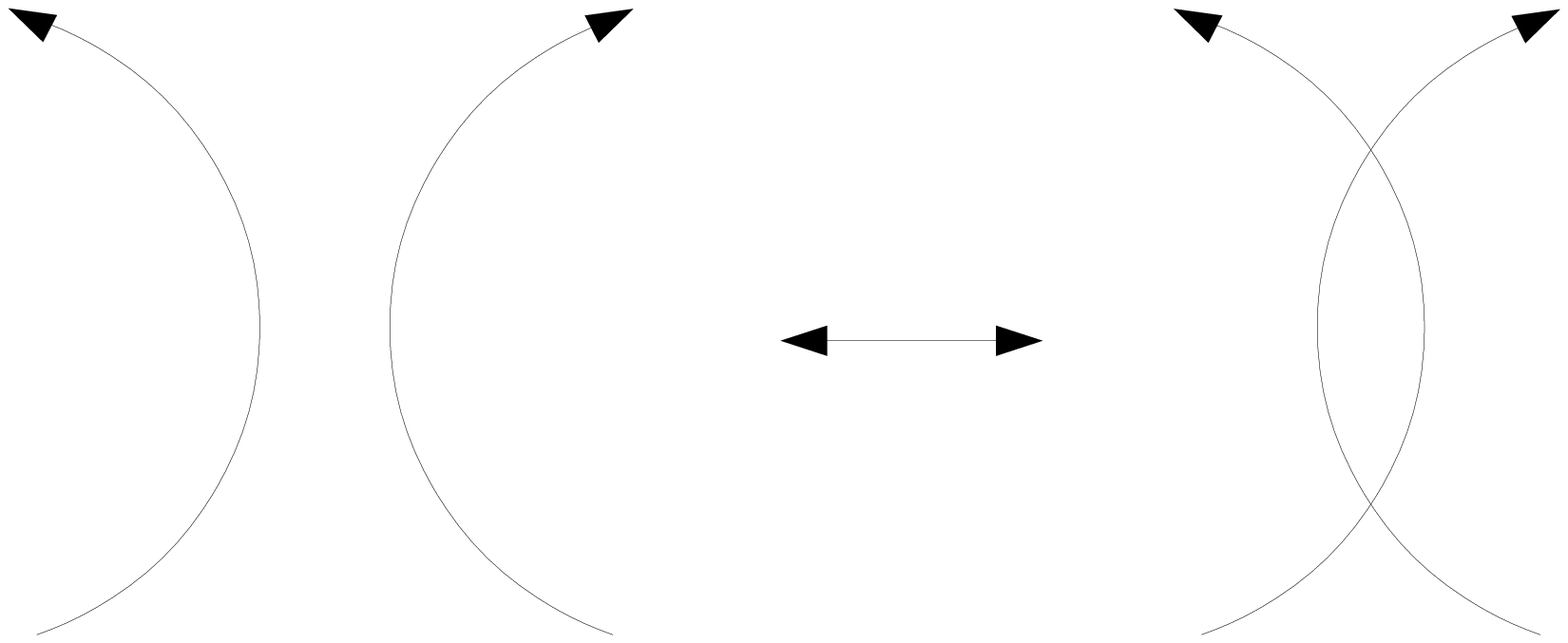}
	\caption{A direct self-tangency move}
\end{subfigure}
\hspace{2cm}
\begin{subfigure}{0.4\textwidth}
	\includegraphics[width=\textwidth, trim = 1cm 18cm 3cm 3cm, clip]{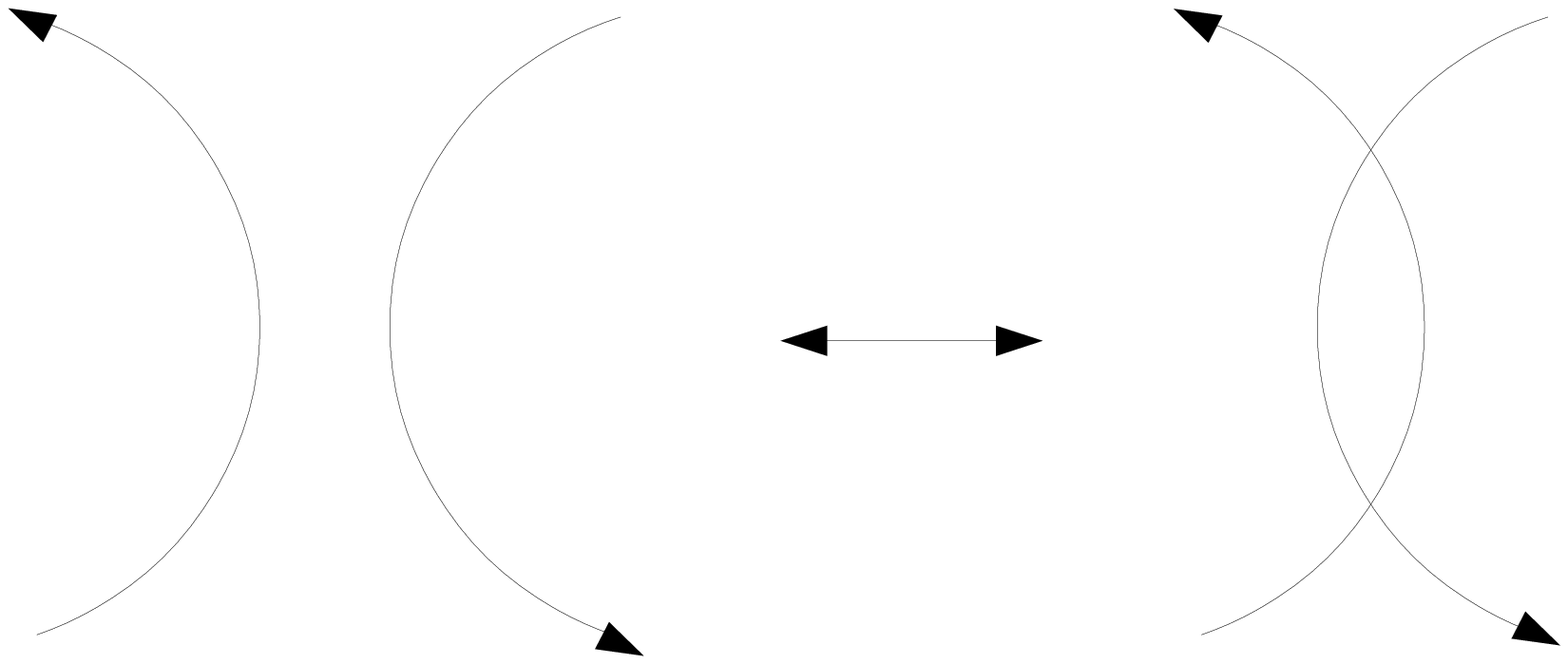}
	\caption{An opposite self-tangency move}
\end{subfigure}
\caption{Direct and opposite self-tangency moves}
\label{st-classification}
\end{figure}

This characterization of $J^+(\Gamma)$ determines it uniquely up to addition of a term depending only on the regular homotopy class of $\Gamma$ (in the class of immersed curves). Arnold specifies a value of $J^+$ on a standard representative of each homotopy class of $\Gamma$, thereby specifying it exactly \cite{j+}.

Lanzat and Polyak's polynomial invariant in \cite{lp} is likewise defined for generic curves in the plane. For such a $\Gamma$, they extend $\mathrm{ind}_{\Gamma}$ to a total function on $\mathbb R^2$ by first noting that $\mathrm{ind}_{\Gamma}$ is locally constant on $\mathbb R^2 \setminus \Gamma$ and then defining $\mathrm{ind}_\Gamma(p)$ for $p \in \Gamma$ to be the average of $\mathrm{ind}_\Gamma$ over the connected components of $\mathbb R^2 \setminus \Gamma$ in a neighborhood of $p$ (See Figure \ref{index}). Then Lanzat and Polyak's invariant $I_q(\Gamma) \in \mathbb R[q^{\frac12}, q^{-\frac12}]$ is given by
\begin{equation}\label{lp}
I_q(\Gamma) = \frac1{2\pi} \left(\int_{S^1} k(t) \cdot q^{\mathrm{ind}_\Gamma(\Gamma(t))} \, dt
- \sum_{d\in X} \theta_d \cdot q^{\mathrm{ind}_\Gamma(d)}(q^{\frac12} - q^{-\frac12})\right)
\end{equation}
where $X$ is the set of double points of $\Gamma$, and for each $X \ni d = \Gamma(t_1) = \Gamma(t_2)$, one defines $\theta_d \in (0,\pi)$ as the unsigned angle between $\Gamma'(t_1)$ and $-\Gamma'(t_2)$.

Substituting $q=1$ into the polynomial immediately gives
\[ I_1(\Gamma) = \frac{1}{2\pi} \int_{S^1} k(t) \, dt = \mathrm{rot}(\Gamma) \]
It is in this sense that Lanzat and Polyak's polynomial is considered to be a quantization or quantum deformation of the rotation number. Lanzat and Polyak show that their polynomial is invariant under regular homotopies in the space of generic curves. They calculate its value on representatives of the regular homotopy classes (with respect to regular homotopies in the space of \emph{all} immersed curves) and how it changes under the different kinds of self-tangency and triple-point moves. From these results, they easily show the relation
\[ I_1'(\Gamma) = \frac12 (1 - J^+(\Gamma)) \]
where $I_1'(\Gamma)$ is the linear term in the Taylor expansion of $I_q(\Gamma)$ at $q=1$. From this relation and Equation (\ref{lp}) they obtain an integral expression for $J^+(\Gamma)$.

\begin{figure}
\centering
\includegraphics[width = .33\textwidth, trim = 5.5cm 17cm 10cm 5cm, clip]{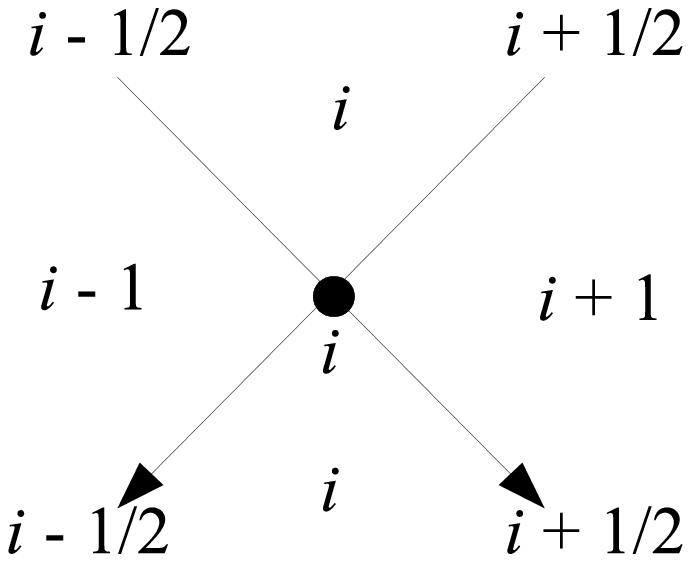}
\caption{A double point of index $i$}
\label{index}
\end{figure}

Both the rotation number and the $J^+$ invariant have been extended to curves in more general surfaces. For the rotation number, first note that the rotation number on the plane can be defined as the unique homomorphism from the group of regular homotopy classes of curves immersed in the plane to the group $\mathbb Z$, subject to the constraint that a small counterclockwise loop should map to 1. To determine a rotation number for a curve in a surface $S$, we first need to specify a set of smooth generators of the fundamental group of $S$; then there is a unique homomorphism from the group of regular homotopy classes of curves in $S$ to $\mathbb Z/|\chi(S)|\mathbb Z$, subject to the constraints that the representatives of the fundamental group map to 0 and that a small counterclockwise contractible loop maps to 1 \cite{mc}. Here $\chi(S)$ is the Euler characteristic of $S$, and $\mathbb Z/|\chi(S)|\mathbb Z$ is simply $\mathbb Z$ if $\chi(S)$ is 0. The choice of generators does not affect the rotation number of homologically trivial curves, which will be the focus of most of the remainder of the paper. For our purposes, then, we can take the following as the definition of the rotation number: The \emph{rotation number} for homologically trivial curves in a closed oriented surface $S$ is the unique homomorphism from the group of regular homotopy classes of homologically trivial curves in $S$ to $\mathbb Z/|\chi(S)|\mathbb Z$, subject to the constraint that a small counterclockwise contractible loop maps to 1.

Several generalizations of $J^+$ are given by Viro in \cite{giro}. Most significantly for our purposes, he defines an invariant $J^+(\Gamma)$ for generic homologically trivial curves immersed in oriented closed surfaces, which is unchanged by opposite self-tangency moves and triple-point moves, and which under direct self-tangency moves increases by 2.
\begin{rem}
Actually Viro generalizes the related invariant $J^-$, which is unchanged by direct self-tangency moves and triple-point moves, and which decreases by 2 at opposite self-tangencies. As he points out, though, given such an invariant $J^-$ we can immediately define an invariant $J^+$ by $J^+ = J^- + n$, where $n$ is the number of double points of the curve. It is easily checked that $J^+$ defined in this way is unchanged by opposite self-tangency moves and triple-point moves, and increases by 2 under direct self-tangency moves.
\end{rem}

Viro's $J^+$ generalizes Arnold's earlier construction of a $J^+$ invariant for curves on the sphere.
To calculate $SJ^+(\Gamma)$, as Arnold called it, stereographically project the sphere onto the plane, choosing some point in $S^2 \setminus \Gamma$ to be the point at infinity. Then letting $\Gamma'$ be the resulting curve in the plane,
\begin{equation}\label{sj+}
 SJ^+(\Gamma) = J^+(\Gamma') + \frac{\mathrm{rot}(\Gamma')^2}{2}
\end{equation}
Arnold showed that $SJ^+$ is independent of the point we choose as the point at infinity in the stereographic projection, and it clearly changes in the way that a $J^+$ invariant is required to under self-tangency and triple-point moves.

In this paper we attempt to generalize Lanzat and Polyak's polynomial and related results to curves in surfaces, as far as is possible. We introduce a more general index function which makes sense in a general connected oriented surface, but the price for this generality is twofold: The new index function $\mathrm{ind}_{\Gamma,b}$ is defined not only in terms of $\Gamma$ but also a base point $b$ in the surface, and $\Gamma$ must be homologically trivial. Using the $\mathrm{ind}_{\Gamma,b}$ we construct a polynomial $I_q(\Gamma,b)$, which is defined by integrating local geometric data much as Lanzat and Polyak's polynomial is. We show that $I_q(\Gamma,b)$ is independent of the Riemannian metric on the surface. Evaluating the polynomial at $q=1$, we obtain a new integral formula for the rotation number, which is also a generalization of one form of the Gauss-Bonnet theorem. For surfaces with nonzero Euler characteristic, we show that $J^+(\Gamma)$ can be calculated from the value and first derivative of $I_q(\Gamma,b)$ at $q=1$, thereby obtaining an integral formula for $J^+(\Gamma)$. Lastly, we use this formula to give an explicit expression for $SJ^+$.

\section{Generalizing the Index Function}
In a general surface, the concept of how many times a curve goes around a point is meaningless. Note, however, that before extending its domain from $\mathbb R^2 \setminus \Gamma$ to $\mathbb R^2$, $\mathrm{ind}_\Gamma$ is the unique function $\mathbb R^2 \setminus \Gamma \to \mathbb Z$ satisfying the following conditions:
\begin{enumerate}[(i)] 
\item $\mathrm{ind}_\Gamma$ is locally constant.
\item $\mathrm{ind}_\Gamma$ increases by 1 when we make a positive crossing over $\Gamma$.
\item $\mathrm{ind}_\Gamma(p) = 0$ for a point $p$ on the outside of $\Gamma$.
\end{enumerate}
Although the third condition has no meaning in a general surface, the first two can be considered in any oriented surface. For an oriented closed surface $S$ and a generic smooth curve $\Gamma:S^1 \to S$, constructing an integer-valued function on $\mathbb R^2 \setminus \Gamma$ satisfying the conditions (i) and (ii) is equivalent to constructing a singular 2-chain with boundary $\Gamma$. Thus such a function can be constructed if and only if $\Gamma$ is homologically trivial. In this case, if $S$ is connected then the function is uniquely determined up to addition of a constant. This motivates the following definition: For an oriented connected surface $S$, a homologically trivial generic curve $\Gamma:S^1 \to S$, and a base point $b \in S \setminus \Gamma$, let $\mathrm{ind}_{\Gamma,b}$ be the unique function $S \setminus \Gamma \to \mathbb Z$ satisfying the following conditions:
\begin{enumerate}[(i)]
\item $\mathrm{ind}_{\Gamma,b}$ is locally constant.
\item $\mathrm{ind}_{\Gamma,b}$ increases by 1 when we make a positive crossing over $\Gamma$.
\item $\mathrm{ind}_{\Gamma,b}(b) = 0$.
\end{enumerate}
\begin{rem}
It can be easily checked that $\mathrm{ind}_{\Gamma,b}(p)$ is the intersection index of any path from $b$ to $p$ with $\Gamma$, which can be taken as an alternative definition of $\mathrm{ind}_{\Gamma,b}$. In this case the requirement that $S$ is connected ensures that such a path exists, and the requirement that $\Gamma$ is homologically trivial ensures that the intersection index does not depend on the path chosen.
\end{rem}
As before, we extend $\mathrm{ind}_{\Gamma, b}$ to all of $S$ by saying that for $p$ in the image of $\Gamma$, $\mathrm{ind}_{\Gamma, b}(p)$ is found by averaging $\mathrm{ind}_{\Gamma, b}$ over the connected components of $S \setminus \Gamma$ in a neighborhood of $S$.

\section{Main Result}
Let $S$ be an oriented connected closed surface with Riemannian metric, let $\Gamma: S^1 \to S$ be a homologically trivial generic smooth curve on $S$, and let $b \in S \setminus \Gamma$. Let $X \subseteq S$ be the set of double points of $\Gamma$. For each double point $d \in X$, let $\theta_d$ be the unsigned angle between the tangent vectors $\Gamma'(t_1)$ and $-\Gamma'(t_2)$, where $\{t_1, t_2\} = \Gamma^{-1}(\{d\})$. Let $k_g(t)$ be the geodesic curvature of $\Gamma$ at $\Gamma(t)$, and let $K: S \to \mathbb R$ be the Gaussian curvature. Define $I_q(\Gamma, b)$ as
\[ \frac1{2\pi} \left(
\int_{S_1} k_g(t) \cdot q^{\mathrm{ind}_{\Gamma, b}(\Gamma(t))} \, dt
- \sum_{d \in X} \theta_d \cdot q^{\mathrm{ind}_{\Gamma, b}(d)}(q^{\frac12}-q^{-\frac12})
+ \iint_S K \cdot \frac{q^{\mathrm{ind}_{\Gamma, b}}-1}{q^{\frac12}-q^{-\frac12}} \, dA
\right) \]

\begin{thm} $I_q(\Gamma, b)$ is preserved under orientation-preserving diffeomorphisms. That is, $I_q(\Gamma, b)$ does not depend on the Riemannian metric on $S$.
\end{thm}
\begin{proof}
For $j \in \frac12 \mathbb Z \setminus \mathbb Z$, let $S_j$ be the subsurface of $S$ where the index is greater than $j$. Observe that $\partial S_j$ is a piecewise smooth curve whose pieces are the arcs of $\Gamma$ with index $j$. The orientation which $\partial S_j$ inherits from $\Gamma$ agrees with its orientation as the boundary of $S_j$. For each $i \in \mathbb Z$ let $X_i \subseteq S$ be the set of double points with degree $i$. At each $d \in X_{j-\frac12}$, $\partial S_j$ changes direction by $\pi-\theta_d$. At each $d \in X_{j-\frac12}$, $\partial S_j$ changes direction by $-(\pi-\theta_d)$. By the Gauss-Bonnet Theorem,
\[ 2\pi\chi(S_j) =
\iint_{S_j} K\,dA
+\int_{S^1} k_g(t) \cdot \mathbf1_{\mathrm{ind}_{\Gamma, b}} \, dt
+ \sum_{d \in X_{j - \frac12}}(\pi - \theta_d)
- \sum_{d \in X_{j + \frac12}}(\pi - \theta_d)
\]
and so
\[\sum_{j \in \frac12 \mathbb Z \setminus \mathbb Z} 2\pi \chi(S_j) q^j =\]
\[\iint_{S} K\sum_{i=0}^\infty q^{\mathrm{ind}_{\Gamma, b} - i - \frac12}\,dA
+\int_{S^1} k_g(t) \cdot q^{\mathrm{ind}_{\Gamma, b}(\Gamma(t))} \, dt
+\sum_{d \in X}(\pi - \theta_d)(q^{\mathrm{ind}_{\Gamma, b}(d) + \frac12} - q^{\mathrm{ind}_{\Gamma, b}(d) - \frac12})\]
\[ = \iint_{S} K\frac{q^{\mathrm{ind}_{\Gamma,b}}}{q^{\frac12} - q^{-\frac12}}\,dA
+\int_{S^1} k_g(t) \cdot q^{\mathrm{ind}_{\Gamma, b}(\Gamma(t))} \, dt
+\sum_{d \in X}(\pi - \theta_d)(q^{\mathrm{ind}_{\Gamma, b}(d) + \frac12} - q^{\mathrm{ind}_{\Gamma, b}(d) - \frac12}) \]
in $C^\infty((1,+\infty))$. Applying the Gauss-Bonnet Theorem to all of $S$ gives $2\pi \chi(S) = \iint_S K \,dA$, so
\[\sum_{j \in \frac12 \mathbb Z \setminus \mathbb Z} 2\pi \chi(S_j) q^j - \frac{2\pi\chi(S)}{q^{\frac12} - q^{-\frac12}}
- \sum_{d \in X}\pi(q^{\mathrm{ind}_{\Gamma, b}(d) + \frac12} - q^{\mathrm{ind}_{\Gamma, b}(d) - \frac12})  \]
\[ = \iint_S K \cdot \frac{q^{\mathrm{ind}_{\Gamma, b}}-1}{q^{\frac12}-q^{-\frac12}} \, dA 
+\int_{S_1} k_g(t) \cdot q^{\mathrm{ind}_{\Gamma, b}(\Gamma(t))} \, dt
- \sum_{d \in X} \theta_d \cdot q^{\mathrm{ind}_{\Gamma, b}(d)}(q^{\frac12}-q^{-\frac12})
\]
\[ = 2\pi I_q(\Gamma,b) \]
Thus
\begin{equation}\label{topological-formula}
 I_q(\Gamma,b) =
\sum_{j \in \frac12 \mathbb Z \setminus \mathbb Z} \chi(S_j) q^j 
- \frac{\chi(S)}{q^{\frac12} - q^{-\frac12}}
- \frac12\sum_{d \in X}(q^{\mathrm{ind}_{\Gamma, b}(d) + \frac12} - q^{\mathrm{ind}_{\Gamma, b}(d) - \frac12})
\end{equation}
None of the terms on the right of (\ref{topological-formula}) depend on the Riemannian metric, so $I_q(\Gamma,b)$ is preserved under orientation-preserving diffeomorphisms.
\end{proof}
\begin{rem}
Although two of the terms on the right side of Equation (\ref{topological-formula}) diverge as $q \to 1$, $I_q(\Gamma,b)$ is a polynomial in $q^{\frac12}$ and $q^{-\frac12}$ (as is clear from its integral definition) and is thus defined for all $q>0$.
\end{rem}
\begin{rem}
It is not quite true that $I_q(\Gamma,b)$ is invariant under regular homotopies of $\Gamma$ in the space of generic curves; such a homotopy might change the position of $\Gamma$ relative to the base point $b$. However, if we view $\Gamma$ together with the choice of base point as a generic immersion $S^1 \sqcup \{\bullet\} \to S$, it follows immediately from the theorem that $I_q$ is invariant under regular homotopies in the space of such generic immersions.
\end{rem}

\section{Behavior under change of base point}
For a new base point $b'$, $\mathrm{ind}_{\Gamma, b'}$ and $\mathrm{ind}_{\Gamma, b}$ differ by a constant. Suppose $\mathrm{ind}_{\Gamma, b'}=\mathrm{ind}_{\Gamma, b} + C$. Then 
\begin{multline}\label{basepoint-change}
 I_q(\Gamma, b')
= \frac1{2\pi}\left(
\int_{S_1} k_g(t) \cdot q^{\mathrm{ind}_{\Gamma, b}(\Gamma(t))+C} \, dt
- \sum_{d \in X} \theta_d \cdot q^{\mathrm{ind}_{\Gamma, b}(d)+C}(q^{\frac12}-q^{-\frac12})\right.\\
\left.+ \iint_S K \cdot \frac{q^{\mathrm{ind}_{\Gamma, b}+C}-1}{q^{\frac12}-q^{-\frac12}} \, dA
\right)
\end{multline}
\[ = q^C I_q(\Gamma, b)
+ \frac1{2\pi} \iint_S K \cdot \frac{q^C - 1}{q^{\frac12}-q^{-\frac12}} \, dA\]
\[=  q^C I_q(\Gamma, b) + \frac{q^C - 1}{q^{\frac12}-q^{-\frac12}} \chi(S) \]

\section{Relation to rotation number}
Since Lanzat and Polyak's polynomial $I_q(\Gamma)$ is a quantum deformation of the rotation number, a natural question is what $I_q(\Gamma,b)$ is a quantum deformation of.
Calculating $I_1(\Gamma,b)$ explicitly by substituting $q=1$ into the integral expression of $I_q(\Gamma,b)$, we get
\begin{equation}\label{GB-multiplicities}
 I_1(\Gamma,b) = \frac1{2\pi} \left( \int_{S^1} k_g(t)\,dt + \iint_S K \cdot \mathrm{ind}_{\Gamma,b} \,dA \right)
\end{equation}
In the special case where $\Gamma$ is the boundary of a disk $D \subset S$ and the base point $b$ is on the outside of $D$, we have
\[ \frac1{2\pi} \left( \int_{S^1} k_g(t)\,dt + \iint_S K \cdot \mathrm{ind}_{\Gamma,b} \,dA \right)
= \frac1{2\pi} \left( \int_{\partial D} k_g \, ds + \iint_D K \, dA\right) = 1,\]
one form of the Gauss-Bonnet theorem. Thus Equation (\ref{GB-multiplicities}) is a ``Gauss-Bonnet theorem with multiplicities'' in which $\Gamma$ may wrap around multiple times. But what exactly does $I_1(\Gamma,b)$ measure? By the above intuitive reasoning and by analogy with Lanzat and Polyak's results, we should hope that $I_1(\Gamma, b)$ will coincide with the generalized definition of the rotation number. This indeed turns out to be the case, as we will demonstrate shortly.

McIntyre and Cairns give a formula for the rotation number for generic immersed curves (which they call \emph{normal}) in \cite{mc}. In the case where the curve is homologically trivial, their formula reduces to the following:
\begin{thm}[\cite{mc}]
For each positive (respectively negative) integer $i$, let $S'_i \subseteq S$ be the region where $\mathrm{ind}_{\Gamma, b}$ is greater than or equal to (respectively less than or equal to) $i$. Then the rotation number of $\Gamma$ is given by
\[ \sum_{i > 0} \chi(S'_i) - \sum_{i < 0} \chi(S'_i) \]
if $S$ is a torus, and
\[ \sum_{i > 0} \chi(S'_i) - \sum_{i < 0} \chi(S'_i) \mod |\chi(S)| \]
otherwise.
\end{thm}

Now observe that
\[ S'_i = \left\{ \begin{array}{ll} S \setminus S_{i + \frac12} & i<0 \\ S_{i - \frac12} & i>0 \end{array} \right. \]
so
\[ \sum_{i > 0} \chi(S'_i) - \sum_{i < 0} \chi(S'_i)
= \sum_{i > 0} \chi(S_{i - \frac12}) - \sum_{i < 0} \chi(S \setminus S_{i + \frac12}) \]
\[= \sum_{i > 0} \chi(S_{i - \frac12}) - \sum_{i < 0} (\chi(S) - \chi(S_{i + \frac12}))
= \sum_{j \in \frac12 \mathbb Z \setminus \mathbb Z} (\chi(S_j) - \mathbf 1_{j<0} \cdot \chi(S))\]
By Equation (\ref{topological-formula}),
\[ I_q(\Gamma, b)
= \sum_{j \in \frac12 \mathbb Z \setminus \mathbb Z} \chi(S_j) q^j 
- \frac{\chi(S)}{q^{\frac12} - q^{-\frac12}}
- \frac12\sum_{d \in X}(q^{\mathrm{ind}_{\Gamma, b}(d) + \frac12} - q^{\mathrm{ind}_{\Gamma, b}(d) - \frac12}) \]
\[ = \sum_{j \in \frac12 \mathbb Z \setminus \mathbb Z}(\chi(S_j)- \chi(S))q^j - \frac12\sum_{d \in X}(q^{\mathrm{ind}_{\Gamma, b}(d) + \frac12} - q^{\mathrm{ind}_{\Gamma, b}(d) - \frac12}) \]
so evaluating $I_q(\Gamma, b)$ at $q=1$ gives
\[ I_1(\Gamma,b) =  \sum_{j \in \frac12 \mathbb Z \setminus \mathbb Z} (\chi(S_j) - \mathbf 1_{j<0} \cdot \chi(S))
= \sum_{i > 0} \chi(S'_i) - \sum_{i < 0} \chi(S'_i)\]
Thus,
\begin{prop}
The rotation number of $\Gamma$ is given by
$ I_1(\Gamma,b) $
if $S$ is a torus, and
$ I_1(\Gamma,b) \mod |\chi(S)| $
otherwise.
\end{prop}
\begin{rem}
Using the formula for the change of base point, Equation (\ref{basepoint-change}), and plugging in $q=1$, we see that if $\mathrm{ind}_{\Gamma,b'} = \mathrm{ind}_{\Gamma,b} + C$,
\[ I_1(\Gamma, b') = I_1(\Gamma, b) + C\chi(S). \]
Thus $I_1(\Gamma, b)$ may depend on the choice of $b$ but $I_1(\Gamma, b) \mod |\chi(S)|$ does not.
\end{rem}

\section{Relation to the integral with respect to Euler characteristic}
Let $\tilde \Gamma$ be the smoothing of $\Gamma$, and let $\tilde S_j \subseteq S$ be the region where $\mathrm{ind}_{\tilde\Gamma,b} > 0$ for all $j \in \frac12 \mathbb Z \setminus \mathbb Z$. The Euler characteristic of the region where $\mathrm{ind}_{\tilde\Gamma,b}=i$ is
\[ \chi(\tilde S_{i -\frac12} \setminus \tilde S_{i+\frac12})
= \chi(\tilde S_{i -\frac12})- \chi(\tilde S_{i+\frac12})
= \chi(S_{i -\frac12})- \chi(S_{i+\frac12})
= a_{i -\frac12}- a_{i+\frac12} + \delta_{i,0} \cdot \chi(S)\]
so
\[ \int_{S \setminus \tilde \Gamma} q^{\mathrm{ind}_{\tilde\Gamma,b}} \, d\chi
= \chi(S) + (q^{\frac12} - q^{-\frac12}) \sum_{j \in \frac12 \mathbb Z \setminus \mathbb Z} a_j\]
\[= \chi(S) + (q^{\frac12} - q^{-\frac12}) \left( I_q(\Gamma, b) + \frac12 \sum_{d\in X} (q^\frac12 - q^{-\frac12})q^{\mathrm{ind}_{\Gamma,b}(d)} \right) \]
Solving for $I_q(\Gamma, b)$ yields
\begin{equation}\label{euler} I_q(\Gamma, b) \\
 = -\frac12 \sum_{d\in X} (q^\frac12 - q^{-\frac12})q^{\mathrm{ind}_{\Gamma,b}(d)}
+ \int_{S \setminus \tilde \Gamma} \frac{q^{\mathrm{ind}_{\tilde\Gamma,b}}-1}{q^{\frac12} - q^{-\frac12}} \, d\chi
\end{equation}

\section{Relation to $J^+$ invariants}
In order to connect our results to Viro's generalization of $J^+$, it is necessary to further generalize our definition of an index function. Let a \emph{rational index function} for $\Gamma$ be any function $S \setminus \Gamma \to \mathbb Q$ which is locally constant and which increases by 1 when we make a positive crossing over $\Gamma$. Clearly any two rational index functions for $\Gamma$ differ by a rational constant. For any rational index function $\iota$ for $\Gamma$ we can extend $\iota$ to a function on $S$ by averaging over adjacent regions as before, and then define $I_q(\iota)$ by replacing $\mathrm{ind}_{\Gamma, b}$ with $\iota$ in the definition of $I_q(\Gamma,b)$. (Note that $I_q(\iota)$ no longer need belong to $\mathbb R[q^{\frac12}, q^{-\frac12}]$, but is a member of $\mathbb R[q^t: t\in \mathbb Q]$.) For any rational index function $\iota$ for $\Gamma$ there is a unique rational index function $\tilde\iota$ for $\tilde\Gamma$ such that $\iota$ and $\tilde\iota$ agree away from the curve. The formula for $I_q$ as an integral with respect to the Euler characteristic and the formula for the change of base point generalize in the obvious way. Explicitly,
\begin{equation}\label{eulerQ} I_q(\iota) \\
 = -\frac12 \sum_{d\in X} (q^\frac12 - q^{-\frac12})q^{\iota(d)}
+ \int_{S \setminus \tilde \Gamma} \frac{q^{\tilde\iota}-1}{q^{\frac12} - q^{-\frac12}} \, d\chi
\end{equation}
and for any two rational index functions $\iota, \iota'$ with $\iota'-\iota = C$,
\[ I_q(\iota') =  q^C I_q(\iota) + \frac{q^C - 1}{q^{\frac12}-q^{-\frac12}} \chi(S) \]

\begin{rem}
For any $\zeta$ in the relative homology group $H_2(S, \Gamma; \mathbb Q)$ such that $\partial(\zeta)$ is the fundamental class of $\Gamma$, one can define $\mathrm{ind}_\zeta(x)$ for $x \in S \setminus \Gamma$ as follows: $\mathrm{ind}_\zeta(x)$ is the image of $\zeta$ under the map
\[ H_2(S, \Gamma; \mathbb Q) \to H_2(S, S \setminus x; \mathbb Q) \to \mathbb Q \]
where the map $H_2(S, \Gamma; \mathbb Q) \to H_2(S, S \setminus x; \mathbb Q)$ is the relativization homomorphism and the map $H_2(S, S \setminus x; \mathbb Q) \to \mathbb Q$ is the canonical isomorphism. It is easily checked that $\zeta \mapsto \mathrm{ind}_\zeta$ gives an isomorphism between elements of $H_2(S, \Gamma; \mathbb Q)$ with boundary $\Gamma$ and rational index functions for $\Gamma$, so results about general rational index functions may be rephrased as results about functions of the form $\mathrm{ind}_\zeta$; indeed, it is the latter formulation that Viro uses to define $J^-$. Here we present his results using the language of rational index functions for greater congruence with the rest of the paper.
\end{rem}

For $\chi(S) \neq 0$, Viro defines $J^-(\Gamma)$ as follows: Find the unique rational index function $\tilde\iota_0$ for $\tilde \Gamma$ such that
\[ \int_{S \setminus \tilde\Gamma} \tilde\iota_0 \, d \chi = 0\]
(Note that the condition $\chi(S) \neq 0$ is necessary here to ensure the existence and uniqueness of $\tilde\iota_0$.)
Then
\[ J^-(\Gamma) = 1 - \int_{S \setminus \tilde\Gamma} \tilde\iota_0^2 \, d \chi \]
We can now state the relationship between $I_q$ and $J^+$.
\begin{prop}
Assuming that $\chi(S) \neq 0$,
\[ J^+(\Gamma) = \frac{I_1(\Gamma,b)^2}{\chi(S)} - 2 I_1'(\Gamma,b) + 1 \]
\end{prop}
\begin{proof}
From Equation (\ref{euler}) it immediately follows that
\[ I_1(\Gamma,b) = \int_{S \setminus \tilde \Gamma} (\mathrm{ind}_{\tilde\Gamma,b}) \, d\chi \]
and
\[ I'_1(\Gamma,b) = -\frac{|X|}{2} + \frac12 \int_{S \setminus \tilde \Gamma}  (\mathrm{ind}_{\tilde\Gamma,b})^2 \, d\chi \]
Then $\mathrm{ind}_{\tilde\Gamma,b} - I_1(\Gamma,b)/\chi(S)$ is a rational index function for $\tilde\Gamma$ and
\[ \int_{S \setminus \tilde \Gamma} \left(\mathrm{ind}_{\tilde\Gamma,b} - \frac{I_1(\Gamma,b)}{\chi(S)}\right) \, d\chi
= I_1(\Gamma,b) - I_1(\Gamma,b) = 0\]
so $\tilde\iota_0 = \mathrm{ind}_{\tilde\Gamma,b} - I_1(\Gamma,b)/\chi(S)$ is the unique rational index function for $\tilde\Gamma$ satisfying $\int_{S \setminus \tilde\Gamma} \tilde\iota_0 \, d \chi = 0$.
Now
\[ J^-(\Gamma)
= 1 - \int_{S \setminus \tilde\Gamma} \left(\mathrm{ind}_{\tilde\Gamma,b} - \frac{I_1(\iota)}{\chi(S)}\right)^2 \, d \chi\]
\[= 1 - \int_{S \setminus \tilde\Gamma} (\mathrm{ind}_{\tilde\Gamma,b})^2 \, d \chi 
+ \frac{2 I_1(\Gamma,b)}{\chi(S)}\int_{S \setminus \tilde\Gamma} (\mathrm{ind}_{\tilde\Gamma,b}) \, d \chi
- \frac{I_1(\Gamma,b)^2}{\chi(S)^2} \int_{S \setminus \tilde\Gamma} d \chi   \]
\[ = 1 - (2I'_1(\Gamma,b) + |X|) + \frac{2I_1(\Gamma,b)^2}{\chi(S)} - \frac{I_1(\Gamma,b)^2}{\chi(S)} \]
\[= \frac{I_1(\Gamma,b)^2}{\chi(S)} - 2I_1'(\Gamma,b) + 1 - |X|\]
and so
\[J^+(\Gamma) = J^-(\Gamma) + |X| = \frac{I_1(\Gamma,b)^2}{\chi(S)} - 2I_1'(\Gamma,b) + 1\]
\end{proof}

Plugging in the integral definition of $I_q(\Gamma,b)$, we get
\begin{multline}\label{integral-j+}
J^+(\Gamma) = 
\frac{1}{4\pi^2\chi(S)}\left( \int_{S^1} k_g(t)\,dt + \iint_S \mathrm{ind}_{\Gamma,b} \,dA \right)^2 \\
-\frac1\pi\left(\int_{S^1}k_g(t)\cdot\mathrm{ind}_{\Gamma, b}(\Gamma(t))\,dt
- \sum_{d \in X} \theta_d
+ \frac12\iint_S K \cdot (\mathrm{ind}_{\Gamma, b})^2 \, dA\right)
+ 1.
\end{multline}

\section{A formula for the $SJ^+$ invariant}
Let $\Gamma$ be any curve on the unit sphere. (Note that $\Gamma$ is automatically homologically trivial.) 
Using the facts that $K=1$ on $S^2$ and $\chi(S^2) = 2$ to simplify the expression in Equation (\ref{integral-j+}) gives the following expression for the $SJ^+$ invariant:
\begin{multline*}
 SJ^+(\Gamma) = 
\frac{1}{8\pi^2}\left( \int_{S^1} k_g(t)\,dt + \iint_S \mathrm{ind}_{\Gamma,b} \,dA \right)^2 \\
-\frac1\pi\left(\int_{S^1}k_g(t)\cdot\mathrm{ind}_{\Gamma, b}(\Gamma(t))\,dt
- \sum_{d \in X} \theta_d
+ \frac12\iint_S(\mathrm{ind}_{\Gamma, b})^2 \, dA\right)
+ 1
\end{multline*}

\section{Acknowledgments}
The main results of this paper were obtained in the 2014 Knots \& Graphs summer undergraduate research group at the Ohio State University. I am grateful to the Ohio State University for funding the program, to fellow members of the research group for discussions, and especially to Dr. Sergei Chmutov for invaluable guidance throughout the research process.

\end{document}